\documentclass[12pt,reqno]{amsart}
\usepackage{amssymb}
\usepackage{amscd}
\usepackage{hyperref}

\usepackage{amsxtra}

\usepackage{hyperref}

\usepackage{esint}

\usepackage[shortlabels]{enumitem}
\usepackage{cleveref}

\crefrangeformat{enumi}{items #3#1#4--#5#2#6}

\theoremstyle{plain}

\newtheorem{thm}{Theorem}[section]
\newtheorem{lem}[thm]{Lemma}
\newtheorem{cor}[thm]{Corollary}

\theoremstyle{definition}

\newtheorem{defn}[thm]{Definition}

\theoremstyle{remark}

\newtheorem{rem}[thm]{Remark}

\crefname{thm}{theorem}{theorems}
\crefname{lem}{lemma}{lemmas}
\crefname{cor}{corollary}{corollaries}
\crefname{prop}{proposition}{propositions}
\crefname{mainthm}{theorem}{theorems}
\crefname{maincor}{corollary}{corollaries}

\crefname{defn}{definition}{definitions}
\crefname{conj}{conjecture}{conjectures}
\crefname{example}{example}{examples}
\crefname{exercise}{exercise}{exercises}
\crefname{prob}{problem}{problems}
\crefname{quest}{question}{questions}

\crefname{rem}{remark}{remarks}
\crefname{claim}{claim}{claims}
\crefname{axiom}{axiom}{axioms}
\crefname{hyp}{hypothesis}{hypotheses}
\crefname{notation}{notation}{notations}
\crefname{case}{case}{cases}

\numberwithin{equation}{section}

\usepackage{color}
\definecolor{darkgreen}{cmyk}{1,0,1,.2}
\definecolor{m}{rgb}{1,0.1,1}

\newdimen\theight
\def\TeXref#1{%
             \leavevmode\vadjust{\setbox0=\hbox{{\tt
                     \quad\quad  {\small \textrm #1}}}%
             \theight=\ht0
             \advance\theight by \lineskip
             \kern -\theight \vbox to
             \theight{\rightline{\rlap{\box0}}%
             \vss}%
             }}%

\begin{document}
\vskip .5cm
\title[Liouville Type Theorem for $(\mathcal F,\mathcal F')_{p}$-Harmonic Maps on Foliations]
{Liouville Type Theorem for $(\mathcal F,\mathcal F')_{p}$-Harmonic Maps \\ on Foliations}

\author[X. S. Fu]{Xueshan Fu}
\address{Department of Mathematics\\
         Jeju National University\\
         Jeju 63243\\
         Republic of Korea}
\email{xsfu@jejunu.ac.kr}

\author[S. D. Jung]{Seoung Dal Jung}
\address{Department of Mathematics\\
         Jeju National University\\
         Jeju 63243\\
         Republic of Korea}
\email{sdjung@jejunu.ac.kr}

\thanks{The second author was supported by the National Research Foundation of Korea(NRF) grant funded by the Korea government (MSIP) (NRF-2022R1A2C1003278).}


\begin{abstract}
In this paper, we study $(\mathcal F,\mathcal F')_{p}$-harmonic maps between foliated Riemannian manifolds $(M,g,\mathcal F)$ and $(M',g',\mathcal F')$. A $(\mathcal F,\mathcal F')_{p}$-harmonic map $\phi:(M,g,\mathcal F)\to (M', g',\mathcal F')$ is a critical point of the transversal $p$-energy functional $E_{B,p}$.  Trivially, $(\mathcal F,\mathcal F')_2$-harmonic map is $(\mathcal F,\mathcal F')$-harmonic map, which is a critical point of $E_B$.  There is another definition of a harmonic map on foliated Riemannian manifolds, called transversally harmonic map, which is a solution of the Euler-Largrange equation $\tau_b(\phi)=0$.  Two definitions are not equivalent, but if  $\mathcal F$ is minimal, then two definitons are equivalent.  Firstly, we give the first and second variational formulas for $(\mathcal F,\mathcal F')_{p}$-harmonic maps. Next, we  investigate the generalized Weitzenb\"ock type formula and the Liouville type theorem for $(\mathcal F,\mathcal F')_{p}$-harmonic map. 
\end{abstract}
\maketitle

\section{Introduction}
Let $(M,g)$ and $(M',g')$ be Riemannian manifolds and $\phi:(M,g)\to (M', g')$ be a smooth map. Then $\phi$ is said to be {\it harmonic} if the tension field $\tau(\phi)={\rm tr}_{g}(\nabla d\phi)$ vanishes, equivalently, $\phi$ is a critical point of the energy functional defined by
\begin{align*}
E(\phi)={1\over 2}\int_{M} | d\phi|^2\mu_{M},
\end{align*}
where $\mu_{M}$ is the volume element of $M$ \cite{ES}. In 2003, J. Konderak and R. Wolak introduced the notion of transversally harmonic maps between foliated Riemannian manifolds \cite{KW1}. Let $(M,g,\mathcal F)$ and $(M',g',\mathcal F')$ be foliated Riemannian manifolds and $\phi:(M,g,\mathcal F)\to (M', g',\mathcal F')$ be a smooth foliated map, (i.e., $\phi$ is a smooth leaf-preserving map). Let $Q$ be the normal bundle of $\mathcal F$ and $d_{T}\phi=d\phi|_{Q}$.    Then $\phi$ is said to be {\it transversally harmonic} if  $\phi$ is a solution of the Eular-Largrange equation $\tau_{b}(\phi)=0$, where  $\tau_b(\phi)={\rm tr}_{Q}(\nabla_{\rm tr} d_T\phi)$ is the transversal tension field of $\mathcal F$. Transversally harmonic maps on foliated Riemannian manifolds have been studied by many authors \cite{CZ,KW1,KW2,OSU}. However, a transversally harmonic map is not a critical point of the transversal energy functional \cite{JJ2} defined by
\begin{align*}
E_{B}(\phi)=\frac{1}{2}\int_{M} | d_T \phi|^2\mu_{M}.
\end{align*}
In 2013, S. Dragomir and A. Tommasoli \cite{DT} defined a new harmonic map, called {\it $(\mathcal F,\mathcal F')$-harmonic map}, which is a critical point of the transversal energy functional $E_{B}$. Two definitions are equivalent when $\mathcal F$ is minimal. Analogously to $p$-harmonic map on ordinary manifolds,  we define $(\mathcal F,\mathcal F')_{p}$-harmonic map, $p>1$, which is a generalization of $(\mathcal F,\mathcal F')$-harmonic map.
In fact, a smooth foliated map $\phi$ is said to be {\it $(\mathcal F,\mathcal F')_{p}$-harmonic} if $\phi$ is a critical point of the {\it transversal $p$-energy functional} defined by
\begin{align*}
E_{B,p}(\phi)={1\over p}\int_{M} | d_T \phi|^p\mu_{M}.
\end{align*}
Trivially, $(\mathcal F,\mathcal F')_{2}$-harmonic map is just $(\mathcal F,\mathcal F')$-harmonic map \cite{DT} and $(\mathcal F,\mathcal F')_{p}$-harmonic maps  are $p$-harmonic maps for point foliations. 
Similarly, we can define the {\it transversally $p$-harmonic map}  as a solution of the Euler-Largrange equation $\tau_{b,p}(\phi)=0$, where $\tau_{b,p}(\phi)= {\rm tr}_{Q}(\nabla_{\rm tr} |d_T\phi|^{p-2}d_T\phi)$ is the  transversal $p$-tension field of $\mathcal F$ (cf. \ref{eq3-3}). 
Two definitions are not equivalent if $\mathcal F$ is not minimal (cf. Theorem \ref{th3}).  But if $\mathcal F$ is minimal,  then two definitions are equivalent, and so they are generalizations of $p$-harmonic maps on an ordinary manifold.    In this paper, we study the Liouville type  theorem for $(\mathcal F,\mathcal F')_p$-harmonic maps not for transversally $p$-harmonic maps.  But  the Liouville type theorem for transversally $2$-harmonic maps have been studied  in \cite{FJ,JJ}.  The Liouville type theorem on Riemannian manifolds has been studied by many researchers \cite{Jung1,mlj,NN,EN,sy,t,y}. 
   In particular, D.J. Moon  et al. \cite{mlj}  proved  the Liouville type theorem for p-harmonic maps on  a complete Riemannian manifold as follows.  

\begin{thm}  \cite{mlj} Let $(M,g)$ be a complete Riemannian manifold and let $(M',g')$ be a Riemannian manifold of non-positive sectional curvature. Assume that the Ricci curvature ${\rm Ric^M}$ of $M$ satisfies ${\rm Ric}^M \geq - {4(p-1)\over p^2 }\mu_0$ and $> - {4(p-1)\over p^2}\mu_0$ at some point, where $\mu_0$ is the infimum of the eigenvalues of the Laplacian acting on $L^2$-functions on $M$. Then any $p$-harmonic map $\phi:M\to M'$  of finite $p$-energy is constant.
\end{thm}
We generalize Theorem 1.1 to $(\mathcal F,\mathcal F')_p$-harmonic maps on foliated manifolds.
This paper is organized as follows. In Section 2, we review some basic facts on foliated Riemannian manifolds. In Section 3, we prove the first variational formula for the transversal $p$-energy functional (Theorem \ref{th3}). From the first variational formula, we know that the transversal $p$-harmonic map is not a critical point of the transversal $p$-energy functional. And  the second variational formulas for  $(\mathcal F,\mathcal F')_p$-harmonic map is given (Theorem \ref{th4}) and the stability of $(\mathcal F,\mathcal F')_p$-harmonic map is studied. 
In Section 4, we prove the generalized Weitzenb\"ock type formula  for $(\mathcal F,\mathcal F')_p$-harmonic map (Corollary \ref{co3}).  As an  application of the Weitzenb\"ock type formula,  we prove the Liouville type theorem for $(\mathcal F,\mathcal F')_{p}$-harmonic maps (Theorem \ref{th5}).  Namely, let $\lambda_0$ be the infimum of the eigenvalues of the basic Laplacian acting on $L^2$-basic functions on $(M,\mathcal F)$. 

\begin{thm} (cf. Theorem 4.5)
Let $(M,g,\mathcal F)$ be a complete foliated Riemannian manifold with coclosed mean curvature form and all leaves be compact.
Let $(M',g',\mathcal F')$ be a foliated Riemannian manifold with non-positive transversal sectional curvature. Assume that the transversal Ricci curvature ${\rm Ric^{Q}}$ of $M$ satisfies ${\rm Ric^{Q}}\geq-\frac{4(p-1)}{p^{2}}\lambda_{0}$  and ${\rm Ric^{Q}}>-\frac{4(p-1)}{p^{2}}\lambda_{0}$ at some point $x_{0}$.
Then any $(\mathcal F,\mathcal F')_{p}$-harmonic map $\phi : (M,g,\mathcal F) \rightarrow (M', g',\mathcal F')$ of $E_{B,p}(\phi)<\infty$ is transversally constant.
\end{thm}
In particular, under the  same assumptions of $(M,\mathcal F)$  as in Theorem 1.2, if ${\rm Ric}^Q \geq -{4(p-1)\over p^2}\lambda_0$ and ${\rm Ric}^Q> -{4(p-1)\over p^2}\lambda_0$ at some point, then  for any $q$ with $2\leq q \leq p$,  every $(\mathcal F,\mathcal F')_q$-harmonic map   of finite transversal $q$-energy  $E_{B,q}(\phi)$ is transversally constant (Corollary \ref{co5}).

\section{Preliminaries}

Let $(M,g,\mathcal F)$ be a foliated Riemannian
manifold with a foliation $\mathcal F$ of codimension $q$ and a bundle-like metric $g$ with respect to $\mathcal F$ \cite{Molino,Tond}.  Let $TM$ be the tangent bundle of $M$, $T\mathcal F$
the tangent bundle of $\mathcal F$, and  $Q=TM/T\mathcal F$ the
corresponding normal bundle of $\mathcal F$.
Let $g_Q$ be the holonomy invariant metric on
$Q$ induced by $g$.
This means that $L_Xg_Q=0$ for $X\in T\mathcal F$, where
$L_X$ is the transverse Lie derivative. We denote by $\nabla^Q$ the transverse Levi-Civita
connection on the normal bundle $Q$ \cite{Tond,Tond1}. The transversal curvature tensor $R^Q$ of $\nabla^Q\equiv\nabla$ is defined by $R^Q(X,Y)=[\nabla_X,\nabla_Y]-\nabla_{[X,Y]}$ for any $X,Y\in\Gamma TM$. Let $K^Q$ and ${\rm Ric}^Q $ be the transversal
sectional curvature and transversal Ricci operator with respect to $\nabla$, respectively. The foliation $\mathcal F$ is said to be {\it
minimal} if the mean curvature form $\kappa$ of $\mathcal F$ vanishes, that is,  $\kappa=0$ \cite{Tond}.
Let $\Omega_B^r(\mathcal F)$ be the space of all {\it basic
$r$-forms}, i.e.,  $\omega\in\Omega_B^r(\mathcal F)$ if and only if
$i(X)\omega=0$ and $L_X\omega=0$ for any $X\in\Gamma T\mathcal F$, where $i(X)$ is the interior product. Then $\Omega^*(M)=\Omega_B^*(\mathcal F)\oplus \Omega_B^*(\mathcal F)^\perp$ \cite{Lop}.  Let $\kappa_B$ be the basic part of $\kappa$. Then $\kappa_B$ is closed, i.e., $d\kappa_B=0$ \cite{Lop}.
Let $\bar *:\Omega_B^r(\mathcal F)\to \Omega_B^{q-r}(\mathcal F)$ be the star operator  given by
\begin{align*}
\bar *\omega = (-1)^{(n-q)(q-r)} *(\omega\wedge\chi_{\mathcal F}),\quad \omega\in\Omega_B^r(\mathcal F),
\end{align*}
where $\chi_{\mathcal F}$ is the characteristic form of $\mathcal F$ and $*$ is the Hodge star operator associated to $g$.  Let $\langle\cdot,\cdot\rangle$ be the pointwise inner product on $\Omega_B^r(\mathcal F)$, which is given by
\begin{align*}
\langle\omega_1,\omega_2\rangle \nu = \omega_1\wedge\bar * \omega_2,
\end{align*}
where $\nu$ is the transversal volume form such that $*\nu =\chi_{\mathcal F}$. 
 Let $\delta_B :\Omega_B^r (\mathcal F)\to \Omega_B^{r-1}(\mathcal F)$ be the operator defined by
\begin{align*}
\delta_B\omega = (-1)^{q(r+1)+1} \bar * (d_B-\kappa_B \wedge) \bar *\omega,
\end{align*}
where $d_B = d|_{\Omega_B^*(\mathcal F)}$. It is well known  \cite{Park} that $\delta_B$ is the formal adjoint of $d_B$ with respect to the global inner product $\ll\cdot,\cdot\gg$, which is defined by
\begin{align}\label{2-1}
\ll \omega_1,\omega_2\gg =\int_M \langle\omega_1,\omega_2\rangle\mu_M, 
\end{align}
where $\mu_M =\nu\wedge\chi_{\mathcal F}$ is the volume form.
 The  basic
Laplacian $\Delta_B$ acting on $\Omega_B^*(\mathcal F)$ is given by
\begin{equation*}
\Delta_B=d_B\delta_B+\delta_B d_B.
\end{equation*}
 Let $\{E_a\}_{a=1,\cdots,q}$ be a local orthonormal basic frame on $Q$.  Then  $\delta_{B}$ is  locally expressed by 
\begin{equation}\label{2-2}
\delta_{B} = -\sum_a i(E_a) \nabla_{E_a} + i (\kappa_{B}^\sharp),
\end{equation}
 where $(\cdot)^\sharp$ is the dual vector field of $(\cdot)$. 
 Let  $Y$ be a transversal infinitesimal automorphism,  i.e., $[Y,Z]\in \Gamma T\mathcal F$ for
all $Z\in \Gamma T\mathcal F$ \cite{Kamber2}. Let  $\bar Y = \pi (Y)$.  
Then we obtain the transversal divergence theorem
 on a foliated Riemannian
manifold.
\begin{thm} \label{thm1-1} \cite{Yorozu}
Let $(M,g,\mathcal F)$ be a closed, oriented Riemannian manifold
with a transversally oriented foliation $\mathcal F$ and a
bundle-like metric $g$ with respect to $\mathcal F$. Then for a transversal infinitesimal automorphism $X$,
\begin{equation*}
\int_M \operatorname{div_\nabla}(\bar X) \mu_{M}
= \int_M g_Q(\bar X,\kappa_B^\sharp)\mu_{M},
\end{equation*}
where $\operatorname{div_\nabla} \bar{X}$
denotes the transversal divergence of $\bar{X}$ with respect to the
connection $\nabla$.
\end{thm}
Now we define the bundle map $A_Y:\Gamma Q\to \Gamma Q$ for any $Y\in TM$ by
\begin{align}\label{eq1-11}
A_Y s =L_Ys-\nabla_Ys,
\end{align}
where $L_Y s = \pi [Y,Y_s]$ for $\pi(Y_s)=s$. It is well-known \cite{Kamber2} that for any  infitesimal automorphism $Y$ 
\begin{align*}
A_Y s = -\nabla_{Y_s}\bar Y,
\end{align*}
where $Y_s$ is the vector field such that $\pi(Y_s)=s$. So $A_Y$ depends only on $\bar Y=\pi(Y)$ and is a linear operator.  Moreover, $A_Y$ extends in an obvious way to tensors of any type on $Q$  \cite{Kamber2}.
Since
$L_X\omega=\nabla_X\omega$ for any $X\in\Gamma T\mathcal F$, $A_Y$
preserves the basic forms. 
Then we
have the generalized Weitzenb\"ock formula on $\Omega_B^*(\mathcal F)$ \cite{Jung}: for any $\omega\in\Omega_B^r(\mathcal  F),$
\begin{align}\label{2-3}
  \Delta_B \omega = \nabla_{\rm tr}^*\nabla_{\rm tr}\omega +
  F(\omega)+A_{\kappa_B^\sharp}\omega,
\end{align}
where $F(\omega)=\sum_{a,b}\theta^a \wedge i(E_b)R^Q(E_b,
 E_a)\omega$ and 
 \begin{align}\label{2-4}
\nabla_{\rm tr}^*\nabla_{\rm tr}\omega =-\sum_a \nabla^2_{E_a,E_a}\omega
+\nabla_{\kappa_B^\sharp}\omega.
\end{align}
 The operator $\nabla_{\rm tr}^*\nabla_{\rm tr}$
is positive definite and formally self adjoint on the space of
basic forms \cite{Jung}. 
  If $\omega$ is a basic 1-form, then $F(\omega)^\sharp
 ={\rm Ric}^Q(\omega^\sharp)$.

\section{Variational formulas for $(\mathcal F,\mathcal F')_{p}$-harmonic map}

Let $(M,  g,\mathcal F)$  and $(M', g',\mathcal F')$ be two foliated Riemannian manifolds and let $\phi:(M,g,\mathcal F)\to (M', g',\mathcal F')$ be a smooth foliated map,
i.e., $d\phi(T\mathcal F)\subset T\mathcal F'$. We define $d_T\phi:Q \to Q'$ by
\begin{align}
d_T\phi := \pi' \circ d \phi \circ \sigma.
\end{align}
Then $d_T\phi$ is a section in $ Q^*\otimes
\phi^{-1}Q'$, where $\phi^{-1}Q'$ is the pull-back bundle on $M$. Let $\nabla^\phi$
and $\tilde \nabla$ be the connections on $\phi^{-1}Q'$ and
$Q^*\otimes \phi^{-1}Q'$, respectively. Then a foliated map $\phi:(M, g,\mathcal F)\to (M', g',\mathcal F')$ is called {\it transversally totally geodesic} if it satisfies
\begin{align}
\tilde\nabla_{\rm tr}d_T\phi=0,
\end{align}
where $(\tilde\nabla_{\rm tr}d_T\phi)(X,Y)=(\tilde\nabla_X d_T\phi)(Y)$ for any $X,Y\in \Gamma Q$. Note that if $\phi:(M,g,\mathcal F)\to (M',g',\mathcal F')$ is transversally totally geodesic with $d\phi(Q)\subset Q'$, then, for any transversal geodesic $\gamma$ in $M$, $\phi\circ\gamma$ is also transversal geodesic.
From now on, we use $\nabla$ instead of all induced connections if we have no confusion.
The {\it transversal $p$-tension field} $\tau_{b,p}(\phi)$ of $\phi$ is defined by
\begin{align}\label{eq3-3}
\tau_{b,p}(\phi):={\rm tr}_{Q}(\nabla_{\rm tr} (|d_T\phi|^{p-2}d_T\phi)),
\end{align}
where $|d_T\phi|^2=\sum_a g_{Q'}(d_T\phi(E_a),d_T\phi(E_a))$. By a direct calculation, we get
\begin{align*}
\tau_{b,p}(\phi)=|d_T\phi|^{p-2}\{\tau_{b}(\phi)+(p-2)d_T\phi({\rm grad_{Q}}(\ln|d_T\phi|))\},
\end{align*}
where $\tau_{b}(\phi)={\rm tr}_{Q}(\nabla_{\rm tr} d_T\phi)$ is the transversal tension field \cite{JJ2}. It follows that $\tau_{b,2}(\phi)=\tau_{b}(\phi)$.

Let $\Omega$ be a compact domain of $M$. Then the {\it transversal $p$-energy}  of $\phi$ on $\Omega\subset
M$ is defined by
\begin{align}\label{eq2-4}
E_{B,p}(\phi;\Omega)={1\over p}\int_{\Omega} | d_T \phi|^p\mu_{M}.
\end{align}

\begin{defn}
Let $\phi: (M, g,\mathcal F) \to (M', g',\mathcal F')$ be
a smooth foliated map. Then $\phi$ is said to be {\it $(\mathcal F,\mathcal F')_{p}$-harmonic} if $\phi$ is a critical point of the transversal $p$-energy functional $E_{B,p}$.
\end{defn}

In particular, $(\mathcal F,\mathcal F')_{2}$-harmonic map is called $(\mathcal F,\mathcal F')$-harmonic map.

Let $V\in\phi^{-1}Q'$. Then there is a 1-parameter family of foliated maps $\phi_t$ with $\phi_0=\phi$ and ${d\phi_t\over dt}|_{t=0}=V$. The family $\{\phi_t\}$ is said to be a {\it foliated variation} of $\phi$ with the {\it normal variation vector field} $V$. Then we have the first variational formula.

\begin{thm} $(${\rm The first variational formula}$)$ \label{th3}
Let $\phi:(M, g, \mathcal F)\to (M', g', \mathcal F')$
be a smooth foliated map and $\{\phi_t\}$ be a smooth foliated variation of $\phi$ supported in a compact domain $\Omega$. Then
\begin{align}\label{eq2-5}
{d\over dt}E_{B,p}(\phi_t;\Omega)|_{t=0}=-\int_{\Omega} \langle V,\tilde{\tau}_{b,p}(\phi)\rangle \mu_{M},
\end{align}
where $\tilde{\tau}_{b,p}(\phi)=\tau_{b,p}(\phi)-|d_T\phi|^{p-2}d_T\phi(\kappa_B^\sharp),$
$V={d\phi_t\over dt}|_{t=0}$ is the normal variation
vector field of $\{\phi_t\}$ and $\langle\cdot,\cdot\rangle$ is the pull-back metric on $\phi^{-1}Q'$.
\end{thm}

\begin{proof}
Fix $x\in M$. Let $\{E_a\}$ be a local orthonormal basic frame on $Q$ such that $(\nabla E_a)(x)=0$. Define
$\Phi:M \times (-\epsilon,\epsilon) \to M'$ by $\Phi(x,t)=\phi_t (x)$. Then
$d\Phi(E_a)=d_T\phi_t (E_a)$, $d\Phi({\partial\over\partial t})={{d\phi_t}\over {dt}}$ and
$\nabla_{\partial\over {\partial t}} {\partial\over{\partial t}}=\nabla_{\partial\over {\partial t}} E_a=\nabla_{E_a}{\partial\over{\partial t}}=0.$
Hence at $x$,
\begin{align*}
\frac{d}{dt} E_{B,p}(\phi_t;\Omega)
&=\frac{1}{p}\frac{d}{dt}\int_{\Omega}(\sum_a\langle d\Phi(E_a), d\Phi(E_a)\rangle)^{\frac{p}{2}}\mu_{M}\\
&= \int_{\Omega}\sum_a |d_{T}\Phi|^{p-2}\langle\nabla_{\partial\over{\partial t}} d\Phi(E_a), d\Phi(E_a)\rangle\mu_{M} \\
&= \int_{\Omega}\sum_a |d_{T}\Phi|^{p-2}\langle\nabla_{E_a} d\Phi({\partial\over\partial t}), d\Phi(E_a)\rangle\mu_{M}\\
&=\int_{\Omega} \sum_a\{E_a\langle d\Phi(\frac{\partial}{\partial t}), |d_T\Phi|^{p-2}d\Phi (E_a)\rangle - \langle d\Phi(\frac{\partial}{\partial t}), (\nabla_{E_a}|d_T\Phi|^{p-2}d\Phi) (E_a)\rangle \}\mu_{M}\\
&= \int_{\Omega}\sum_a E_a \langle \frac{d\phi_t}{dt}, |d_T\phi_{t}|^{p-2}d_T \phi_t (E_a)\rangle \mu_{M} -\int_{\Omega} \langle \frac{d\phi_t}{d t}, \tau_{b,p}(\phi_t)\rangle\mu_{M},
\end{align*}
where $|d_T\Phi|^2=\sum_{a=1}^{q} \langle d\Phi(E_a),d\Phi(E_a)\rangle=|d_T\phi_{t}|^2.$

If we choose a normal vector field $X_{t}$ with
\begin{align*}
\langle X_{t},Z\rangle = \langle \frac{d\phi_t}{d t}, |d_T\phi_{t}|^{p-2}d_T \phi_t(Z)\rangle
\end{align*}
for any  vector field $Z$, then
\begin{align*}
{\rm div}_\nabla (X_{t})  = \sum_a E_a\langle \frac{d\phi_t}{dt}, |d_T\phi_{t}|^{p-2}d_T \phi_t(E_a)\rangle.
\end{align*}
So by the transversal divergence theorem (Theorem \ref{thm1-1}), we have
\begin{align*}
\frac{d}{dt}E_{B,p}(\phi_t;\Omega)
&= \int_{\Omega} {\rm div}_\nabla (X_{t})\mu_M-\int_{\Omega} \langle \frac{d\phi_t}{dt}, \tau_{b,p}(\phi_t)\rangle\mu_{M}\\
&=\int_\Omega\langle X_{t},\kappa_B^\sharp\rangle\mu_M-\int_{\Omega} \langle \frac{d\phi_t}{dt}, \tau_{b,p}(\phi_t)\rangle\mu_{M}\\
&=-\int_{\Omega}\langle \frac{d\phi_t}{dt}, \tau_{b,p}(\phi_t)-|d_T\phi_{t}|^{p-2}d_T\phi_{t}(\kappa_B^\sharp)\rangle\mu_{M}\\
&=-\int_{\Omega}\langle \frac{d\phi_t}{dt}, \tilde{\tau}_{b,p}(\phi_{t})\rangle\mu_{M},
\end{align*}
which proves (\ref{eq2-5}) by $t=0$.
\end{proof}

\begin{cor}\label{co1}
Let $\phi:(M, g, \mathcal F) \rightarrow (M', g', \mathcal F')$ be a smooth foliated map. Then $\phi$ is $(\mathcal F,\mathcal F')_{p}$-harmonic map if and only if $\tilde{\tau}_{b,p}(\phi)=0$.
\end{cor}

Now, we consider the second variational formula for the transversal $p$-energy.
Let $V,W\in\phi^{-1}Q'$. Then there exists a family of foliated maps $\phi_{t,s}(-\epsilon<s,t<\epsilon)$ satisfying
\begin{align}\label{ee1}
\left\{
  \begin{array}{ll}
    V=\frac{\partial \phi_{t,s}}{\partial t}|_{(t,s)=(0,0)},\\\\
    W=\frac{\partial \phi_{t,s}}{\partial s}|_{(t,s)=(0,0)}, \\\\
    \phi_{0,0}=\phi.
  \end{array}
\right.
\end{align}
The family $\{\phi_{t,s}\}$ is said to be the {\it foliated variation} of $\phi$ with the {\it normal variation vector fields} $V$ and $W$.

\begin{thm} $(${\rm The second variational formula}$)$\label{th4}
Let $\phi:(M, g, \mathcal F)\to (M', g', \mathcal F')$ be a $(\mathcal F,\mathcal F')_{p}$-harmonic map. Then for the normal variation vector fields $V$ and $W$ of the foliated variation $\{\phi_{t,s}\}$,
\begin{align*}
\frac{\partial^{2}}{\partial t\partial s}& E_{B,p}(\phi_{t,s};\Omega)|_{(t,s)=(0,0)}\notag\\
=&\int_{\Omega}|d_T\phi|^{p-2}\langle \nabla_{\rm tr}V, \nabla_{\rm tr}W\rangle \mu_M
-\int_{\Omega} |d_T\phi|^{p-2}\langle {\rm tr_{Q}}R^{Q'}(V, d_T \phi)d_T \phi,W\rangle\mu_M \notag\\
&+(p-2)\int_{\Omega}|d_T\phi|^{p-4}\langle \nabla_{\rm tr}V,d_T \phi\rangle\langle \nabla_{\rm tr}W, d_T \phi\rangle\mu_M,
\end{align*}
where ${\rm tr_{Q}}R^{Q'}(V, d_T \phi)d_T \phi=\sum_a R^{Q'}(V, d_T \phi(E_{a}))d_T \phi(E_{a})$.
\end{thm}

\begin{proof}
Let $\Phi: M\times(-\epsilon, \epsilon)\times(-\epsilon, \epsilon)\rightarrow M'$ be a smooth map which is defined by $\Phi(x,t,s)=\phi_{t,s}(x)$.
Then $d\Phi(E_a)=d_T\phi_{t,s} (E_a)$, $d\Phi(\frac{\partial}{\partial s})=\frac{\partial \phi_{t,s}}{\partial s}$
and $d\Phi(\frac{\partial}{\partial t})=\frac{\partial \phi_{t,s}}{\partial t}$.
Trivially, $[X, \frac{\partial}{\partial t}]=[X, \frac{\partial}{\partial s}]=0$ for any vector field $X\in TM$.
For convenience, we put $f=|d_T\phi_{t,s}|^{p-2}$ and $f_{0}=|d_T\phi|^{p-2}$. By making use of the first variational formula, it turns out that
\begin{align}\label{ee2}
\frac{\partial}{\partial s}E_{B,p}(\phi_{t,s};\Omega)
=-\int_{\Omega}\langle d\Phi(\frac{\partial}{\partial s}), \tilde{\tau}_{b,p}(\phi_{t,s})\rangle \mu_{M}.
\end{align}
Differentiating (\ref{ee2}) with respect to $t$, we get
\begin{align}\label{ee43}
\frac{\partial^{2}}{\partial t \partial s}E_{B,p}(\phi_{t,s};\Omega)
=-\int_{\Omega}\langle \nabla_{\frac{\partial}{\partial t}}d\Phi(\frac{\partial}{\partial s}), \tilde{\tau}_{b,p}(\phi_{t,s})\rangle \mu_{M}
-\int_{\Omega}\langle d\Phi(\frac{\partial}{\partial s}), \nabla_{\frac{\partial}{\partial t}}\tilde{\tau}_{b,p}(\phi_{t,s})\rangle \mu_{M}.
\end{align}
Since $\phi$ is $(\mathcal F,\mathcal F')_{p}$-harmonic map, from Corollary \ref{co1}, we have that
at $(t,s)=(0,0)$,

\begin{align}\label{ee50}
\frac{\partial^{2}}{\partial t \partial s}E_{B,p}(\phi_{t,s};\Omega)|_{(0,0)}
=-\int_{\Omega}\langle W, \nabla_{\frac{\partial}{\partial t}}\tilde{\tau}_{b,p}(\phi_{t,s})|_{(0,0)}\rangle \mu_{M}.
\end{align}
By choosing a local orthonormal basic frame field $E_{a}$ with $\nabla E_{a}(x)=0$ at some point $x\in M$, we have that at $x$,
\begin{align}\label{ee23}
\nabla_{\partial\over{\partial t}}&\tilde{\tau}_{b,p}(\phi_{t,s}) \notag\\
=&\nabla_{\partial\over{\partial t}}\tau_{b,p}(\phi_{t,s})-\nabla_{\partial\over{\partial t}}fd\Phi(\kappa_B^\sharp)\notag\\
=&\sum_a \nabla_{\partial\over{\partial t}}\{(\nabla_{E_{a}} fd\Phi)(E_{a})\}-f\nabla_{\kappa_B^\sharp}d\Phi(\frac{\partial}{\partial t})-\frac{\partial f}{\partial t} d\Phi(\kappa_B^\sharp)\notag\\
=&\sum_a \{\nabla_{E_a}\nabla_{\partial\over{\partial t}}fd\Phi(E_a)+ R^{\Phi}(\frac{\partial}{\partial t}, E_a)fd\Phi(E_a)\}-f\nabla_{\kappa_B^\sharp}d\Phi(\frac{\partial}{\partial t})-\frac{\partial f}{\partial t} d\Phi(\kappa_B^\sharp) \notag\\
=&\sum_a \{ \nabla_{E_a}\nabla_{E_a}fd\Phi(\frac{\partial}{\partial t})
+\nabla_{E_a}(\frac{\partial f}{\partial t}d\Phi(E_a)-E_a(f)d\Phi(\frac{\partial}{\partial t}))+R^{\Phi}(\frac{\partial}{\partial t}, E_a)fd\Phi(E_a)\}\notag\\
&-f\nabla_{\kappa_B^\sharp}d\Phi(\frac{\partial}{\partial t})-\frac{\partial f}{\partial t} d\Phi(\kappa_B^\sharp).
\end{align}

From (\ref{ee23}), we have
\begin{align}\label{ee44}
\int_{\Omega} &\langle \nabla_{\frac{\partial}{\partial t}}\tilde{\tau}_{b,p}(\phi_{t,s}),d\Phi(\frac{\partial}{\partial s})\rangle\mu_M \notag\\
=&\int_{\Omega} \sum_a \langle \nabla_{E_a}\nabla_{E_a}fd\Phi(\frac{\partial}{\partial t}), d\Phi(\frac{\partial}{\partial s})\rangle\mu_M
+\int_{\Omega}\sum_a \langle R^{Q'}(d\Phi(\frac{\partial}{\partial t}), d\Phi(E_a))fd\Phi(E_a),d\Phi(\frac{\partial}{\partial s})\rangle\mu_M \notag\\
&+\int_{\Omega}\sum_a E_a \langle \frac{\partial f}{\partial t}d\Phi(E_a), d\Phi(\frac{\partial}{\partial s}) \rangle\mu_M
-\int_{\Omega}\sum_a \langle \frac{\partial f}{\partial t}d\Phi(E_a), \nabla_{E_a}d\Phi(\frac{\partial}{\partial s}) \rangle\mu_M \notag\\
&-\int_{\Omega}\sum_a E_a \langle E_a(f)d\Phi(\frac{\partial}{\partial t}),d\Phi(\frac{\partial}{\partial s}) \rangle\mu_M
+\int_{\Omega}\sum_a \langle E_a(f)d\Phi(\frac{\partial}{\partial t}), \nabla_{E_a}d\Phi(\frac{\partial}{\partial s}) \rangle\mu_M\notag\\
&-\int_{\Omega}\langle f\nabla_{\kappa_B^\sharp}d\Phi(\frac{\partial}{\partial t}),d\Phi(\frac{\partial}{\partial s})\rangle\mu_M
-\int_{\Omega}\langle\frac{\partial f}{\partial t}d\Phi(\kappa_B^\sharp),d\Phi(\frac{\partial}{\partial s})\rangle\mu_M .
\end{align}
Let $X_{t,s}$ and $Y_{t,s}$ be two normal vector fields such that
\begin{align}\label{ee45}
\left\{
  \begin{array}{ll}
    \langle X_{t,s}, Z\rangle=\langle \frac{\partial f}{\partial t}d\Phi(Z),d\Phi(\frac{\partial}{\partial s})\rangle,\\\\
    \langle Y_{t,s}, Z\rangle=\langle Z(f)d\Phi(\frac{\partial}{\partial t}), d\Phi(\frac{\partial}{\partial s})\rangle
  \end{array}
\right.
\end{align}
for any vector field $Z$ on $M$, respectively.
Then
\begin{align}\label{ee46}
\left\{
  \begin{array}{ll}
  {\rm div}_\nabla (X_{t,s})=\sum_a E_a \langle \frac{\partial f}{\partial t}d\Phi(E_{a}),d\Phi(\frac{\partial}{\partial s})\rangle,\\\\
  {\rm div}_\nabla (Y_{t,s})=\sum_a E_a \langle E_{a}(f)d\Phi(\frac{\partial}{\partial t}), d\Phi(\frac{\partial}{\partial s})\rangle.
  \end{array}
\right.
\end{align}
By (\ref{ee46}) and the transversal divergence theorem (Theorem \ref{thm1-1}), we have
\begin{align}\label{ee47}
\int_{\Omega}\sum_a & E_a \langle \frac{\partial f}{\partial t}d\Phi(E_a), d\Phi(\frac{\partial}{\partial s}) \rangle\mu_M
-\int_{\Omega}\sum_a E_a \langle E_a(f)d\Phi(\frac{\partial}{\partial t}), d\Phi(\frac{\partial}{\partial s}) \rangle\mu_M \notag\\
=&\int_{\Omega} {\rm div_{\nabla}}(X_{t,s})\mu_M -\int_{\Omega}{\rm div_{\nabla}}(Y_{t,s})\mu_M \notag\\
=&\int_{\Omega}\langle \frac{\partial f}{\partial t}d\Phi(\kappa_B^\sharp), d\Phi(\frac{\partial}{\partial s})\rangle\mu_M
-\int_{\Omega}\langle \kappa_B^\sharp(f)d\Phi(\frac{\partial}{\partial t}), d\Phi(\frac{\partial}{\partial s})\rangle\mu_M.
\end{align}
From (\ref{ee50}), (\ref{ee44}) and (\ref{ee47}), we get
\begin{align}\label{ee19}
\frac{\partial^{2}}{\partial t \partial s}& E_{B,p}(\phi_{t,s};\Omega)|_{(0,0)} \notag\\
=&\int_{\Omega}\langle\nabla_{\rm tr}^*\nabla_{\rm tr}f_{0}V, W \rangle \mu_M
-\int_{\Omega}\sum_a f_{0}\langle R^{Q'}(V, d\Phi(E_a))d\Phi(E_a),W\rangle\mu_M \notag\\
&+\int_{\Omega}\sum_a \langle \frac{\partial f}{\partial t}|_{(0,0)}d\Phi(E_a), \nabla_{E_a}W \rangle\mu_M
-\int_{\Omega}\sum_a \langle E_a(f_{0})V, \nabla_{E_a}W \rangle\mu_M \notag\\
=&\int_{\Omega}\sum_a |d_T\phi|^{p-2}\langle \nabla_{E_a}V, \nabla_{E_a}W \rangle \mu_M
-\int_{\Omega}\sum_a |d_T\phi|^{p-2}\langle R^{Q'}(V, d_{T}\phi(E_a))d_{T}\phi(E_a),W\rangle\mu_M \notag\\
&+\int_{\Omega}\sum_a \langle \frac{\partial f}{\partial t}|_{(0,0)}d_{T}\phi(E_a), \nabla_{E_a}W \rangle\mu_M.
\end{align}
Since
\begin{align*}
\frac{\partial f}{\partial t}|_{(0,0)}=(p-2)|d_T\phi|^{p-4}\sum_{b}\langle \nabla_{E_b}V,d_T \phi(E_b)\rangle,
\end{align*}
the proof follows from (\ref{ee19}).
\end{proof}

\begin{cor} \cite{DT} \label{co4}
Let $\phi:(M, g, \mathcal F)\to (M', g', \mathcal F')$ be a $(\mathcal F,\mathcal F')$-harmonic map. Then
\begin{align*}
\frac{\partial^{2}}{\partial t\partial s} E_{B}(\phi_{t,s};\Omega)|_{(t,s)=(0,0)}
=\int_{\Omega}\langle \nabla_{\rm tr}V, \nabla_{\rm tr}W\rangle \mu_M
-\int_{\Omega}\langle {\rm tr_{Q}}R^{Q'}(V, d_T \phi)d_T \phi,W\rangle\mu_M.
\end{align*}
\end{cor}

We define the index form for $(\mathcal F,\mathcal F')_{p}$-harmonic maps by 
\begin{align}\label{em1}
I_{p}(V,W):=\frac{\partial^{2}}{\partial t\partial s}E_{B,p}(\phi_{t,s})|_{(t,s)=(0,0)}
\end{align}
for vector fields $V$ and $W$ along $\phi$.

\begin{rem}
From Theorem \ref{th4} and (\ref{em1}), we obtain $I_{p}(V,W)=I_{p}(W,V)$.
\end{rem}

\begin{defn}\label{de11}
A $(\mathcal F,\mathcal F')_{p}$-harmonic map $\phi$ is said to be {\it transversally stable} if $I_{p}(V,V)\geq0$ for any vector field $V$ along $\phi$.
\end{defn}
It is easy to obtain the following theorem from Theorem \ref{th4}.
\begin{thm}\label{th1}$(${\rm Stability}$)$
Let $\phi: (M, g,\mathcal F) \to (M', g',\mathcal F')$ be a $(\mathcal F,\mathcal F')_{p}$-harmonic map with compact $M$. If the transversal sectional curvature of $M'$ is non-positive, then $\phi$ is transversally stable.
\end{thm}
\begin{proof}
By Theorem \ref{th4}, we have
\begin{align}\label{ee36}
I_{p}(V,V)
=&\int_{M}|d_T\phi|^{p-2}\{|\nabla_{\rm tr}V|^{2}-\langle R^{Q'}(V, d_T \phi)d_T \phi,V\rangle\}\mu_{M} \notag\\
&+(p-2)\int_{M}|d_T\phi|^{p-4}\langle \nabla_{\rm tr}V, d_T\phi\rangle^{2}\mu_{M}.
\end{align}
Since $K^{Q'}\leq0$, from (\ref{ee36}), we get
$$\langle R^{Q'}(V, d_T \phi)d_T \phi,V\rangle=\sum_a\langle R^{Q'}(V, d_T \phi(E_{a}))d_T \phi(E_{a}),V\rangle=\sum_a K^{Q'}(V, d_T \phi(E_{a}))\leq0,$$
which implies $I_{p}(V,V)\geq0$. So the proof follows.
\end{proof}

\section{Liouville type theorem}

Let $\phi :(M,g,\mathcal F) \rightarrow (M', g',\mathcal F')$ be a smooth foliated map and $\Omega_B^r(E)=\Omega_B^r(\mathcal F)\otimes E$ be the space of $E$-valued basic $r$-forms, where $E=\phi^{-1}Q'$. We define $d_\nabla : \Omega_B^r(E)\to \Omega_B^{r+1}(E)$ by
\begin{align}
d_\nabla(\omega\otimes s)=d_B\omega\otimes s+(-1)^r\omega\wedge\nabla s
\end{align}
for any $s\in \Gamma E$ and $\omega\in\Omega_B^r(\mathcal F)$.
Let $\delta_\nabla$ be a formal adjoint of $d_\nabla$ with respect to the inner product induced from (\ref{2-1}).
Then the Laplacian $\Delta$ on $\Omega_B^*(E)$ is defined by
\begin{align}\label{ee8}
\Delta =d_\nabla \delta_\nabla +\delta_\nabla d_\nabla.
\end{align}
Moreover, the operators $A_X$ and $L_X$ are extended to $\Omega_B^r(E)$ as follows:
\begin{align}
A_X(\omega\otimes s)&=A_X\omega\otimes s\\
L_X(\omega\otimes s)&=L_X\omega\otimes s+\omega\otimes\nabla_X s
\end{align}
for any $\omega\otimes s\in\Omega_B^r(E)$ and  $X\in\Gamma TM$. Then $L_X=d_\nabla i(X) +i(X)d_\nabla$ for any $X\in \Gamma TM$, where $i(X)(\omega\otimes s)=i(X)\omega\otimes s$. Hence $\Psi \in\Omega_B^*(E)$ if and only if $i(X)\Psi=0$ and $L_X\Psi=0$ for all $ X\in \Gamma T\mathcal F$.
Then the generalized Weitzenb\"ock type formula (\ref{2-3}) is extended to $\Omega_B^*(E)$ as follows \cite{JJ2}: for any $\Psi\in\Omega_B^r(E)$,
\begin{align}\label{eq4-6}
\Delta \Psi = \nabla_{\rm tr}^*\nabla_{\rm tr} \Psi
 + A_{\kappa_{B}^\sharp} \Psi + F(\Psi), 
\end{align}
where $ \nabla_{\rm tr}^*\nabla_{\rm tr}$ is the operator induced from (\ref{2-4}) and $F(\Psi)=\sum_{a,b=1}^{q}\theta^a\wedge i(E_b) R(E_b,E_a)\Psi$.
Moreover, we have that for any $ \Psi\in\Omega_B^r(E)$,
\begin{align}\label{ee51}
\frac12\Delta_B|\Psi |^{2}
=\langle\Delta \Psi, \Psi\rangle -|\nabla_{\rm tr} \Psi|^2-\langle A_{\kappa_{B}^\sharp}\Psi, \Psi\rangle -\langle F(\Psi),\Psi\rangle.
\end{align}
If we put $\Psi=|d_T\phi|^{p-2}d_T \phi$, then we have the following theorem.

\begin{thm}\label{th2}
Let $\phi:(M, g,\mathcal F) \to (M', g', \mathcal F')$ be a smooth foliated map. Then the generalized Weitzenb\"ock type formula is given by
\begin{align}\label{ee6}
\frac12\Delta_B| d_T \phi |^{2p-2}
=& \langle\Delta |d_T\phi|^{p-2}d_T \phi, |d_T\phi|^{p-2}d_T \phi\rangle -
 |\nabla_{\rm tr} |d_T\phi|^{p-2}d_T \phi|^2 \notag\\
& -\langle A_{\kappa_{B}^\sharp}|d_T\phi|^{p-2}d_T \phi, |d_T\phi|^{p-2}d_T \phi\rangle  -|d_T\phi|^{2p-4}\langle F(d_T\phi),d_T\phi\rangle,
\end{align}
where
\begin{align}\label{ee3}
\langle F(d_T\phi),d_T\phi\rangle&=\sum_a g_{Q'}(d_T \phi({\rm Ric^{Q}}(E_a)),d_T \phi(E_a)) \notag\\
&-\sum_{a,b}g_{Q'}( R^{Q'}(d_T \phi(E_b), d_T \phi(E_a))d_T \phi(E_a), d_T \phi(E_b)).
\end{align}
\end{thm}

\begin{proof}
The equation (\ref{ee3}) follows from  (\cite{JJ2}, Theorem 5.1).
\end{proof}

\begin{lem}\label{lem1}
Let $\phi:(M, g,\mathcal F) \to (M', g', \mathcal F')$ be a $(\mathcal F,\mathcal F')_{p}$-harmonic map. Then
\begin{align*}
\delta_\nabla |d_T\phi|^{p-2}d_T\phi=0.
\end{align*}
\end{lem}

\begin{proof}
Locally, $\delta_\nabla$ is expressed by (\ref{2-2}).  So from Corollary \ref{co1}, it implies that
\begin{align*}
\delta_\nabla |d_T\phi|^{p-2}d_T\phi
=&-\sum_a i(E_a) \nabla_{E_a}|d_T\phi|^{p-2}d_T\phi+i(\kappa_{B}^\sharp)|d_T\phi|^{p-2}d_T\phi\\
=&-\sum_a (\nabla_{E_a}|d_T\phi|^{p-2}d_T\phi)(E_a)+i(\kappa_{B}^\sharp)|d_T\phi|^{p-2}d_T\phi\\
=&-\tau_{b,p} (\phi) +|d_T\phi|^{p-2}i(\kappa_{B}^\sharp)d_T\phi\\
=&-\tilde{\tau}_{b,p}(\phi)\\
=&0.
\end{align*}
\end{proof}

\begin{cor}\label{co3}
Let $\phi:(M, g,\mathcal F) \to (M', g', \mathcal F')$ be a $(\mathcal F,\mathcal F')_{p}$-harmonic map. Then
\begin{align}\label{ee71}
|d_T&\phi|\Delta_B|d_T\phi|^{p-1}
-\langle\delta_\nabla d_\nabla |d_T\phi|^{p-2}d_T \phi, d_T \phi\rangle \notag\\
&+\langle d_\nabla i(\kappa_{B}^\sharp)d_T\phi,|d_T\phi|^{p-2}d_T \phi\rangle
-|d_T\phi|^{p-1}\kappa_{B}^\sharp(|d_T\phi|)\notag\\
&\leq-|d_T\phi|^{p-2}\langle F(d_T\phi),d_T\phi\rangle.
\end{align}
\end{cor}

\begin{proof}
Since $\phi$ is a $(\mathcal F,\mathcal F')_{p}$-harmonic map, from Theorem \ref{th2} and Lemma \ref{lem1}, we have
\begin{align}\label{ee55}
\frac12\Delta_B| d_T \phi |^{2p-2}
=& \langle\delta_\nabla d_\nabla |d_T\phi|^{p-2}d_T \phi, |d_T\phi|^{p-2}d_T \phi\rangle -|\nabla_{\rm tr} |d_T\phi|^{p-2}d_T \phi|^2 \notag\\
&-|d_T\phi|^{p-2}\langle d_\nabla i(\kappa_{B}^\sharp)d_T\phi,|d_T\phi|^{p-2}d_T \phi\rangle
+|d_T\phi|^{2p-3}\kappa_{B}^\sharp(|d_T\phi|)\notag\\
&-|d_T\phi|^{2p-4}\langle F(d_T\phi),d_T\phi\rangle.
\end{align}
By a simple calculation, we have
\begin{align}\label{ee72}
\frac12\Delta_B| d_T \phi |^{2p-2}
=|d_T\phi|^{p-1}\Delta_B|d_T\phi|^{p-1}-|d_{B}|d_T\phi|^{p-1}|^{2}.
\end{align}
From (\ref{ee55}) and (\ref{ee72}), we get
\begin{align}\label{ee73}
|d_T\phi|^{p-1}\Delta_B|d_T\phi|^{p-1}
=&|d_{B}|d_T\phi|^{p-1}|^{2}-|\nabla_{\rm tr}|d_T\phi|^{p-2}d_T\phi|^{2}
+\langle\delta_\nabla d_\nabla |d_T\phi|^{p-2}d_T \phi, |d_T\phi|^{p-2}d_T \phi\rangle \notag\\
&-|d_T\phi|^{p-2}\langle d_\nabla i(\kappa_{B}^\sharp)d_T\phi,|d_T\phi|^{p-2}d_T \phi\rangle
+|d_T\phi|^{2p-3}\kappa_{B}^\sharp(|d_T\phi|)\notag\\
&-|d_T\phi|^{2p-4}\langle F(d_T\phi),d_T\phi\rangle.
\end{align}
By the first Kato's inequality [\ref{BE}], we have
\begin{align}\label{ee74}
|\nabla_{\rm tr}|d_T\phi|^{p-2}d_T\phi|\geq|d_{B}|d_T\phi|^{p-1}|.
\end{align}
Therefore, the result follows from (\ref{ee73}) and (\ref{ee74}).
\end{proof}

The following conclusion is achieved as the application of the generalized Weitzenb\"ock type formula.
\begin{thm}
Let $(M,g,\mathcal F)$ be a closed foliated Riemannian manifold of non-negative transversal Ricci curvature.
Let $(M',g',\mathcal F')$ be a foliated Riemannian manifold of non-positive transversal sectional curvature. If $\phi:(M, g,\mathcal F) \rightarrow (M', g',\mathcal F')$ is a
$(\mathcal F,\mathcal F')_{p}$-harmonic map, then $\phi$ is transversally totally geodesic.
Furthermore, \\
$(1)$ If the transversal Ricci curvature of $\mathcal F$ is
positive somewhere, then $\phi$ is transversally constant.
\\
$(2)$ If the transversal sectional curvature of $\mathcal F'$ is
negative, then $\phi$ is either transversally constant or $\phi(M)$ is a transversally geodesic closed curve.
\end{thm}

\begin{proof}
By the hypothesis and (\ref{ee3}), we know
$\langle F(d_T\phi),d_T\phi\rangle \geq 0.$
Since $\phi$ is a $(\mathcal F,\mathcal F')_{p}$-harmonic map, from Corollary \ref{co3}, we have
\begin{align}\label{ee61}
|d_T\phi|\Delta_B|d_T\phi|^{p-1}
\leq&\langle\delta_\nabla d_\nabla |d_T\phi|^{p-2}d_T \phi, d_T \phi\rangle
-\langle d_\nabla i(\kappa_{B}^\sharp)d_T\phi,|d_T\phi|^{p-2}d_T \phi\rangle \notag\\
&+|d_T\phi|^{p-1}\kappa_{B}^\sharp(|d_T\phi|).
\end{align}
Integrating (\ref{ee61}), we have
\begin{align}\label{ee62}
\int_{M}\langle&|d_T\phi|,\Delta_B|d_T\phi|^{p-1}\rangle\mu_{M}\notag \\
\leq&\int_{M}\langle \delta_\nabla d_\nabla |d_T\phi|^{p-2}d_T\phi,d_T\phi\rangle\mu_{M}
-\int_{M}\langle d_\nabla i(\kappa_{B}^\sharp)d_T\phi,|d_T\phi|^{p-2}d_T \phi\rangle\mu_{M}\notag \\
&+\int_{M}|d_T\phi|^{p-1}\kappa_{B}^\sharp(|d_T\phi|)\mu_{M}.
\end{align}
Since $d_\nabla (d_T\phi)=0$, we get
\begin{align}\label{ee64}
\int_{M}\langle \delta_\nabla d_\nabla |d_T\phi|^{p-2}d_T\phi,d_T\phi\rangle\mu_{M}
=\int_{M}\langle  d_\nabla |d_T\phi|^{p-2}d_T\phi,d_\nabla d_T\phi\rangle\mu_{M}
=0.
\end{align}
Since $\phi$ is a $(\mathcal F,\mathcal F')_{p}$-harmonic map, from Lemma \ref{lem1}, we obtain
\begin{align}\label{ee65}
\int_{M}\langle d_\nabla i(\kappa_{B}^\sharp)d_T\phi,|d_T\phi|^{p-2}d_T \phi\rangle\mu_{M}
=\int_{M}\langle i(\kappa_{B}^\sharp)d_T\phi,\delta_\nabla |d_T\phi|^{p-2}d_T \phi\rangle\mu_{M}
=0.
\end{align}
Now, we choose a bundle-like metric $g$ such that $\delta_{B}\kappa_{B}=0$. Then we have
\begin{align}\label{ee66}
\int_{M}|d_T\phi|^{p-1}\kappa_{B}^\sharp(|d_T\phi|)\mu_{M}
=\frac{1}{p}\int_{M}\kappa_{B}^\sharp(|d_T\phi|^{p})\mu_{M}
=0.
\end{align}
From (\ref{ee62})$\sim$(\ref{ee66}), we get
\begin{align}\label{ee68}
\int_{M}\langle|d_T\phi|,\Delta_B|d_T\phi|^{p-1}\rangle\mu_{M}\leq0.
\end{align}

On the other hand, we know that
\begin{align}\label{ee63}
\int_{M}\langle |d_T\phi|,\Delta_B|d_T\phi|^{p-1}\rangle\mu_{M}
&=\int_{M}\langle d_{B}|d_T\phi|,d_{B}|d_T\phi|^{p-1}\rangle\mu_{M} \notag \\
&=(p-1)\int_{M}|d_T\phi|^{p-2}| d_{B}|d_T\phi||^{2}\mu_{M}\notag \\
&\geq0.
\end{align}
Then from (\ref{ee68}) and (\ref{ee63}), we get
\begin{align}\label{ee67}
0=\int_{M}\langle |d_T\phi|,\Delta_B|d_T\phi|^{p-1}\rangle\mu_{M}=(p-1)\int_{M}|d_T\phi|^{p-2}| d_{B}|d_T\phi||^{2}\mu_{M},
\end{align}
which yields  $d_T\phi=0$ or $ d_{B}|d_T\phi|=0$. If $ d_{B}|d_T\phi|\neq0$, then $d_T\phi=0$, i.e., $\phi$ is transversally constant. Trivially, $\phi$ is transversally totally geodesic.
If $d_T\phi\neq0$, then $ d_{B}|d_T\phi|=0$. It means that $|d_T\phi|$ is constant.
From (\ref{ee73}), we have
\begin{align}\label{ee81}
\langle|d_T\phi|,\Delta_B|d_T\phi|^{p-1}\rangle
=&-|d_T\phi|^{p-2}|\nabla_{\rm tr}d_T\phi|^{2}
-\langle d_\nabla i(\kappa_{B}^\sharp)d_T\phi,|d_T\phi|^{p-2}d_T \phi\rangle \notag \\
&-|d_T\phi|^{p-2}\langle F(d_T\phi),d_T\phi\rangle.
\end{align}
From (\ref{ee67}), (\ref{ee81}) and Lemma \ref{lem1}, we get
\begin{align}\label{ee69}
\int_{M}|d_T\phi|^{p-2}|\nabla_{\rm tr}d_T\phi|^{2}\mu_{M}
+\int_{M}|d_T\phi|^{p-2}\langle F(d_T\phi),d_T\phi\rangle\mu_{M}=0.
\end{align}
Since $|\nabla_{\rm tr}d_T\phi|^{2}\geq0$ and $\langle F(d_T\phi),d_T\phi\rangle \geq 0$,
from (\ref{ee69}), we have
\begin{align}\label{ee11}
|\nabla_{\rm tr} d_T \phi|^2+\langle F(d_T\phi),d_T\phi\rangle=0.
\end{align}
Thus, $\nabla_{\rm tr}d_T\phi=0$, i.e., $\phi$ is transversally totally geodesic.

Furthermore, from (\ref{ee3}) and (\ref{ee11}), we get
\begin{align}\label{ee70}
\left\{
  \begin{array}{ll}
    g_{Q'}(d_T\phi({\rm Ric}^{Q}(E_a)),d_T\phi(E_a))= 0,\\\\
    g_{Q'}(R^{Q'}(d_T\phi(E_a),d_T\phi(E_b))d_T\phi(E_a),d_T\phi(E_b))= 0
  \end{array}
\right.
\end{align}
for any indices $a$ and $b$.
If ${\rm Ric}^{Q}$ is positive at some point, then $d_T\phi=0$, i.e., $\phi$ is transversally constant, which proves (1). For the statement (2), if the rank of $d_T\phi \geq2$, then there exists a point $x\in M$ such that at least two linearly independent vectors at $\phi(x)$, say, $d_T\phi(E_1)$ and $d_T\phi(E_2)$.
Since the transversal sectional curvature $K^{Q'}$ of $\mathcal F'$ is negative,
\begin{align*}
g_{Q'}(R^{Q'}(d_T\phi(E_1),d_T\phi(E_2))d_T\phi(E_2),d_T\phi(E_1))<0,
\end{align*}
which contradicts (\ref{ee70}). Hence the rank of $d_T\phi <2$, that is, the rank of $d_T\phi$ is zero or one everywhere. If the rank of $d_T\phi$ is zero, then $\phi$ is transversally constant. If the rank of $d_T\phi$ is one, then $\phi(M)$ is closed transversally geodesic.
\end{proof}

Next, we investigate the Liouville type theorem for $(\mathcal F,\mathcal F')_{p}$-harmonic map on foliated Riemannian manifold. Let $\mu_{0}$ be the infimum of the eigenvalues of the basic Laplacian $\Delta_{B}$ acting on $L^{2}$-basic functions on $M$. Then the following theorem is obtained.
\begin{thm}\label{th5}
Let $(M,g,\mathcal F)$ be a complete foliated Riemannian manifold with coclosed mean curvature form $\kappa_{B}$ and all leaves be compact.
Let $(M',g',\mathcal F')$ be a foliated Riemannian manifold with non-positive transversal sectional curvature $K^{Q'}$. Assume that the transversal Ricci curvature ${\rm Ric^{Q}}$ of $M$ satisfies ${\rm Ric^{Q}}\geq-\frac{4(p-1)}{p^{2}}\mu_{0}$ for all $x\in M$ and ${\rm Ric^{Q}}>-\frac{4(p-1)}{p^{2}}\mu_{0}$ at some point $x_{0}$.
Then any $(\mathcal F,\mathcal F')_{p}$-harmonic map $\phi : (M,g,\mathcal F) \rightarrow (M', g',\mathcal F')$ of $E_{B,p}(\phi)<\infty$ is transversally constant.
\end{thm}

\begin{proof}
Let $M$ be a complete foliated Riemannian manifold such that ${\rm Ric^{Q}}\geq-C$ for all $x$ and ${\rm Ric^{Q}}>-C$ at some point $x_{0}$, where $C=\frac{4(p-1)}{p^{2}}\mu_{0}$.
Since $K^{Q'}\leq0$ and ${\rm Ric^{Q}}\geq-C$, from (\ref{ee3}), we have
\begin{align}\label{ee13}
\langle F(d_T\phi),d_T\phi\rangle\geq \sum_a g_{Q'}(d_T \phi({\rm Ric^{Q}}(E_a)),d_T \phi(E_a))\geq-C|d_T\phi|^{2}.
\end{align}
Since $\phi$ is a $(\mathcal F,\mathcal F')_{p}$-harmonic map, from Corollary \ref{co3}, we have
\begin{align}\label{ee52}
|d_T&\phi|\Delta_B|d_T\phi|^{p-1}
-\langle\delta_\nabla d_\nabla |d_T\phi|^{p-2}d_T \phi, d_T \phi\rangle \notag\\
&+\langle d_\nabla i(\kappa_{B}^\sharp)d_T\phi,|d_T\phi|^{p-2}d_T \phi\rangle
-|d_T\phi|^{p-1}\kappa_{B}^\sharp(|d_T\phi|)\notag\\
&\leq-|d_T\phi|^{p-2}\sum_a g_{Q'}(d_T \phi({\rm Ric^{Q}}(E_a)),d_T \phi(E_a))\leq C|d_T\phi|^{p}.
\end{align}
Let $B_{l}=\{y\in M|\rho(y)\leq l\}$, where $\rho(y)$ is the distance between leaves through a fixed point $x_{0}$ and $y$.
Let $\omega_{l}$ be the Lipschitz continuous basic function such that
\begin{align*}
\left\{
  \begin{array}{ll}
    0\leq\omega_{l}(y)\leq1 \quad {\rm for \, any} \, y\in M\\
    {\rm supp}\, \omega_{l}\subset B_{2l}\\
    \omega_{l}(y)=1 \quad {\rm for \, any} \,  y\in B_{l}\\
    \lim\limits_{l\rightarrow\infty}\omega_{l}=1\\
    |d\omega_{l}|\leq\frac{\alpha}{l} \quad {\rm almost \, everywhere \, on} M,
  \end{array}
\right.
\end{align*}
where $\alpha$ is positive constant \cite{Y1}. Therefore, $\omega_{l}\phi$ has compact support for any basic form $\phi\in\Omega_{B}^{*}(\mathcal F)$.
Multiplying (\ref{ee52}) by $\omega_{l}^{2}$ and integrating by parts, this yields
\begin{align}\label{ee4}
\int_{M}&\langle \omega_{l}^{2}|d_T\phi|,\Delta_B|d_T\phi|^{p-1}\rangle\mu_{M}
-\int_{M}\langle \omega_{l}^{2}d_T\phi, \delta_\nabla d_\nabla |d_T\phi|^{p-2}d_T\phi\rangle\mu_{M}\notag \\
&+\int_{M}\langle d_\nabla i(\kappa_{B}^\sharp)d_T\phi,\omega_{l}^{2}|d_T\phi|^{p-2}d_T \phi\rangle\mu_{M}
-\int_{M}\langle \omega_{l}^{2}|d_T\phi|^{p-1},\kappa_{B}^\sharp(|d_T\phi|)\rangle\mu_{M}\notag \\
&\leq -\sum_a \int_{M}\omega_{l}^{2} |d_T\phi|^{p-2}g_{Q'}(d_T \phi({\rm Ric^{Q}}(E_a)),d_T \phi(E_a))\mu_{M} \notag \\
&\leq C\int_{M}\omega_{l}^{2}|d_T\phi|^{p}\mu_{M}.
\end{align}
By Lemma \ref{lem1}, since
\begin{align*}
\delta_\nabla (\omega_{l}^{2}|d_T\phi|^{p-2}d_T \phi)=-i(d_{B}\omega_{l}^{2})|d_T\phi|^{p-2}d_T \phi=-2\omega_{l}i(d_{B}\omega_{l})|d_T\phi|^{p-2}d_T \phi,
\end{align*}
we have
\begin{align*}
\bigg{|}\int_{M}\langle d_\nabla i(\kappa_{B}^\sharp)d_T\phi,\omega_{l}^{2}|d_T\phi|^{p-2}d_T \phi\rangle\mu_{M}\bigg{|}
=&\bigg{|}\int_{M}\langle i(\kappa_{B}^\sharp)d_T\phi, -2\omega_{l}i(d_{B}\omega_{l})|d_T\phi|^{p-2}d_T \phi\rangle\mu_{M}\bigg{|}\notag \\
\leq&2\int_{M}\omega_{l}|i(\kappa_{B}^\sharp)d_T\phi||i(d_{B}\omega_{l})|d_T\phi|^{p-2}d_T \phi|\mu_{M}\notag \\
\leq&2\int_{M}\omega_{l}|\kappa_{B}||d_{B}\omega_{l}||d_T\phi|^{p}\mu_{M}\notag \\
\leq&2\frac{\alpha}{l}\max\{|\kappa_{B}|\}\int_{M}\omega_{l}|d_T\phi|^{p}\mu_{M}.
\end{align*}
For the second inequality in the above, we use the fact
\begin{align}\label{ee82}
|X^{\flat}\wedge d_T\phi|^{2}+|i(X)d_T\phi|^{2}=|X|^{2}|d_T\phi|^{2}
\end{align}
for any vector $X$. If we let $l\rightarrow\infty$, then
\begin{align}\label{ee57}
\lim\limits_{l\rightarrow\infty}\int_{M}\langle d_\nabla i(\kappa_{B}^\sharp)d_T\phi,\omega_{l}^{2}|d_T\phi|^{p-2}d_T \phi\rangle\mu_{M}=0.
\end{align}
At the same time, from $\delta_{B}\kappa_{B}=0$ and the Cauchy-Schwarz inequality, we get
\begin{align*}
\bigg{|}\int_{M}\langle \omega_{l}^{2}|d_T\phi|^{p-1},\kappa_{B}^{\sharp}(|d_T\phi|)\rangle\mu_{M}\bigg{|}
=&\frac{1}{p}\bigg{|}\int_{M}\kappa_{B}^{\sharp}(\omega_{l}^{2}|d_T\phi|^{p})\mu_{M}-\int_{M}2\omega_{l}|d_T\phi|^{p}\langle \kappa_{B},d_{B}\omega_{l}\rangle\mu_{M} \bigg{|} \notag \\
=&\frac{2}{p}\bigg{|}\int_{M}\omega_{l}|d_T\phi|^{p}\langle \kappa_{B},d_{B}\omega_{l}\rangle\mu_{M}\bigg{|}\notag \\
\leq&\frac{\alpha}{l}\max\{|\kappa_{B}|\}\int_{M}\omega_{l}|d_T\phi|^{p}\mu_{M}.
\end{align*}
So by letting $l\rightarrow\infty$,
\begin{align}\label{ee35}
\lim\limits_{l\rightarrow\infty}\int_{M}\langle \omega_{l}^{2}|d_T\phi|^{p-1},\kappa_{B}(|d_T\phi|)\rangle\mu_{M}=0.
\end{align}
By the Cauchy-Schwarz inequality, we know that
\begin{align}\label{ee16}
\int_{M}\langle& \omega_{l}^{2}|d_T\phi|,\Delta_B|d_T\phi|^{p-1}\rangle\mu_{M}\notag\\
&=\int_{M}\langle d_{B}(\omega_{l}^{2}|d_T\phi|),d_{B}|d_T\phi|^{p-1}\rangle\mu_{M} \notag\\
&=\frac{A_{1}}{p}\int_{M}\omega_{l}^{2}|d_{B}|d_T\phi|^{\frac{p}{2}}|^{2}\mu_{M}+A_{1}\int_{M}\langle |d_T\phi|^{\frac{p}{2}}d_{B}\omega_{l},\omega_{l}d_{B}|d_T\phi|^{\frac{p}{2}}\rangle\mu_{M} \notag\\
&\geq\frac{A_{1}}{p}\int_{M}\omega_{l}^{2}|d_{B}|d_T\phi|^{\frac{p}{2}}|^{2}\mu_{M}-A_{1}\int_{M}\omega_{l}|d_T\phi|^{\frac{p}{2}}|d_{B}\omega_{l}||d_{B}|d_T\phi|^{\frac{p}{2}}|\mu_{M},
\end{align}
where $A_{1}=\frac{4(p-1)}{p}$.
It is well known \cite{prs} that for a basic function $f$ on $M$, we get from (\ref{ee82})
\begin{align}\label{ee5}
|d_{\nabla}(fd_{T}\phi)|=|d_{B}f\wedge d_{T}\phi|\leq |d_{B}f||d_{T}\phi|.
\end{align}
Hence we have
\begin{align}\label{ee7}
\bigg{|}\int_{M}\langle \omega_{l}^{2}d_T\phi, \delta_\nabla d_\nabla |d_T\phi|^{p-2}d_T\phi\rangle\mu_{M}\bigg{|}
&=\bigg{|}\int_{M}\langle d_\nabla(\omega_{l}^{2}d_T\phi), d_\nabla |d_T\phi|^{p-2}d_T\phi\rangle\mu_{M}\bigg{|} \notag \\
&\leq \int_{M}|d_\nabla(\omega_{l}^{2}d_T\phi)||d_\nabla |d_T\phi|^{p-2}d_T\phi|\mu_{M} \notag \\
&\leq 2\int_{M}|\omega_{l}d_{B}\omega_{l}||d_{B}|d_T\phi|^{p-2}||d_T\phi|^{2}  \notag \\
&\leq A_{2}\int_{M} \omega_{l}|d_{B}\omega_{l}||d_T\phi|^{\frac{p}{2}}|d_{B}|d_T\phi|^{\frac{p}{2}}|,
\end{align}
where $A_{2}=\frac{4(p-2)}{p}$.
From (\ref{ee16}) and (\ref{ee7}), we get
\begin{align}\label{ee10}
\int_{M}&\langle \omega_{l}^{2}|d_T\phi|,\Delta_B|d_T\phi|^{p-1}\rangle\mu_{M}
-\int_{M}\langle \omega_{l}^{2}d_T\phi, \delta_\nabla d_\nabla |d_T\phi|^{p-2}d_T\phi\rangle\mu_{M} \notag \\
&\geq -(A_{1}+A_{2})\int_{M} \omega_{l}|d_{B}\omega_{l}||d_T\phi|^{\frac{p}{2}}|d_{B}|d_T\phi|^{\frac{p}{2}}|\mu_{M}
+\frac{A_{1}}{p}\int_{M}\omega_{l}^{2}|d_{B}|d_T\phi|^{\frac{p}{2}}|^{2}\mu_{M}.
\end{align}
From (\ref{ee4}) and Fatou's inequality, it is trivial that $d_{B}|d_T\phi|^{\frac{p}{2}}\in L^{2}$.
Hence by the H${\rm \ddot{o}}$lder inequality,
$$\int_{M}\omega_{l}|d_{B}\omega_{l}||d_T\phi|^{\frac{p}{2}}|d_{B}|d_T\phi|^{\frac{p}{2}}|\mu_{M}\leq
(\int_{M}|d_T\phi|^{p}|d_{B}\omega_{l}|^{2}\mu_{M})^{\frac{1}{2}}(\int_{M}\omega_{l}^{2}|d_{B}|d_T\phi|^{\frac{p}{2}}|^{2}\mu_{M})^{\frac{1}{2}}.$$
If we let $l\rightarrow\infty$, then
\begin{align}\label{ee32}
\lim\limits_{l\rightarrow\infty}\int_{M}\omega_{l}|d_T\phi|^{\frac{p}{2}}|d_{B}\omega_{l}||d_{B}|d_T\phi|^{\frac{p}{2}}|\mu_{M}=0.
\end{align}
From (\ref{ee16}) and (\ref{ee32}), we have
\begin{align}\label{ee83}
\lim\limits_{l\rightarrow\infty}\int_{M}\langle& \omega_{l}^{2}|d_T\phi|,\Delta_B|d_T\phi|^{p-1}\rangle\mu_{M}
\geq\frac{A_{1}}{p}\int_{M}|d_{B}|d_T\phi|^{\frac{p}{2}}|^{2}\mu_{M}.
\end{align}

On the other hand, by the Rayleigh quotient theorem, we have
\begin{align}\label{ee22}
\frac{\int_{M}\langle d_{B}|d_T\phi|^{\frac{p}{2}}, d_{B}|d_T\phi|^{\frac{p}{2}}\rangle\mu_{M}}{\int_{M} |d_T\phi|^{p}\mu_{M}}\geq\mu_{0}.
\end{align}
From (\ref{ee4}), (\ref{ee57}), (\ref{ee35}), (\ref{ee32}), (\ref{ee83}) and (\ref{ee22}),
by $l\rightarrow\infty$, we get
\begin{align}\label{ee17}
\frac{A_{1}}{p}\mu_{0}\int_{M}|d_T\phi|^{p}\mu_{M}
&\leq\frac{A_{1}}{p}\int_{M}|d_{B}|d_T\phi|^{\frac{p}{2}}|^{2}\mu_{M} \notag\\
&\leq-\sum_a \int_{M} |d_T\phi|^{p-2}g_{Q'}(d_T \phi({\rm Ric^{Q}}(E_a)),d_T \phi(E_a))\mu_{M} \notag\\
&\leq C\int_{M}|d_T\phi|^{p}\mu_{M}.
\end{align}
Since $C=\frac{A_{1}}{p}\mu_{0}$, the equation (\ref{ee17}) implies that
\begin{align}\label{ee18}
\sum_a \int_{M}|d_T\phi|^{p-2}g_{Q'}(d_T \phi({\rm (Ric^{Q}}+C)(E_a)),d_T \phi(E_a))\mu_{M}=0.
\end{align}
Since ${\rm Ric^{Q}}>-C$ at some point $x_{0}$, then $d_T \phi=0$ by (\ref{ee18}).
It means that $\phi$ is transversally constant.
\end{proof}

\begin{rem}
Theorem \ref{th5} can be found for the point foliation in \cite{mlj},
\end{rem}

\begin{cor}\label{co5}
Let $(M,g,\mathcal F)$ be a complete foliated Riemannian manifold with coclosed mean curvature form $\kappa_{B}$ and all leaves be compact.
Let $(M',g',\mathcal F')$ be a foliated Riemannian manifold with non-positive transversal sectional curvature $K^{Q'}$. Assume that the transversal Ricci curvature ${\rm Ric^{Q}}$ of $M$ satisfies ${\rm Ric^{Q}}\geq-\frac{4(p-1)}{p^{2}}\mu_{0}$ for all $x\in M$ and ${\rm Ric^{Q}}>-\frac{4(p-1)}{p^{2}}\mu_{0}$ at some point $x_{0}$.
Then any $(\mathcal F,\mathcal F')_{q}$-harmonic map $\phi : (M,g,\mathcal F) \rightarrow (M', g',\mathcal F')$ with $2\leq q\leq p$ of $E_{B,q}(\phi)<\infty$ is transversally  constant.
\end{cor}

\begin{proof}
For $2\leq q\leq p$, we have $\frac{4(q-1)}{q^{2}}\geq\frac{4(p-1)}{p^{2}}$. So the proof is trivial.
\end{proof}

The following corollary can be obtained easily when $p=2$. 
\begin{cor}\label{co6}
Let $(M,g,\mathcal F)$ be a complete foliated Riemannian manifold with coclosed mean curvature form $\kappa_{B}$ and all leaves be compact.
Let $(M',g',\mathcal F')$ be a foliated Riemannian manifold with non-positive transversal sectional curvature $K^{Q'}$. Assume that the transversal Ricci curvature ${\rm Ric^{Q}}$ of $M$ satisfies ${\rm Ric^{Q}}\geq-\mu_{0}$ for all $x\in M$ and ${\rm Ric^{Q}}>-\mu_{0}$ at some point $x_{0}$.
Then any $(\mathcal F,\mathcal F')$-harmonic map $\phi : (M,g,\mathcal F) \rightarrow (M', g',\mathcal F')$ of $E_{B}(\phi)<\infty$ is transversally  constant.
\end{cor}

\begin{rem}  Corollary \ref{co6} for the transversal harmonic map has been studied by Fu and Jung \cite{FJ}. And if $\mathcal F$ is minimal, then  Liouville type theorem (Theorem \ref{th5}) holds for the transversal $p$-harmonic map.  But we do not know whether Theorem \ref{th5} holds for the transversal $p$-harmonic map $(p>2)$ on arbitrary foliated Riemannian manifolds. 
\end{rem}

\end{document}